\documentclass[10pt]{article}
\usepackage{amsmath,amsthm,amsfonts,latexsym,amsopn,verbatim,amscd,amssymb}
\theoremstyle{plain}
\newtheorem{theorem}{Theorem}[section]
\newtheorem{lemma}[theorem]{Lemma}
\newtheorem{corollary}[theorem]{Corollary}

\newtheorem{proposition}[theorem]{Proposition}

\newcommand{\bnum}{\begin{enumerate}}
\newcommand{\enum}{\end{enumerate}}

\numberwithin{equation}{section}

\DeclareMathOperator{\Aut}{Aut}
\DeclareMathOperator{\Inn}{Inn}
\DeclareMathOperator{\orb}{orb}
\DeclareMathOperator{\cl}{cl}
\begin{document}

\title{\textbf{Autocommuting probability of a finite group relative to its subgroups}}
\author{Parama Dutta and Rajat Kanti Nath\footnote{Corresponding author}}
\date{}
\maketitle
\begin{center}\small{\it
Department of Mathematical Sciences, Tezpur University,\\ Napaam-784028, Sonitpur, Assam, India.\\



Emails:\, parama@gonitsora.com and rajatkantinath@yahoo.com}
\end{center}

\medskip

\begin{abstract}
Let $H \subseteq K$ be two subgroups of a finite group $G$  and $\Aut(K)$ the automorphism group of  $K$. The  autocommuting probability of $G$ relative to its subgroups $H$ and $K$, denoted by ${\Pr}(H, \Aut(K))$, is the probability that  the autocommutator of a randomly chosen pair of elements, one from $H$ and the other from $\Aut(K)$, is equal to the identity element of $G$. In this paper, we study ${\Pr}(H, \Aut(K))$  through a generalization. 
\end{abstract}

\medskip

\noindent {\small{\textit{Key words:} Automorphism group, Autocommuting probability, Autoisoclinism.}


\medskip

\section{Introduction}
Let $G$ be a finite group acting on a set $\Omega$. Let $\Pr(G, \Omega)$ denote the probability that  a randomly chosen element of $\Omega$  fixes a  randomly chosen element of $G$. In 1975, Sherman \cite{sherman} initiated the study of $\Pr(G, \Omega)$ considering $G$ to be an abelian group and  $\Omega = \Aut(G)$, the automorphism group of $G$. Note that
\[
{\Pr}(G,\Aut(G)) = \frac {\left|\{(x,\alpha)\in G\times \Aut(G):[x,\alpha]=1\}\right|}{|G||\Aut(G)|}
\]
where $[x,\alpha]$  is the autocommutator of $x$ and $\alpha$ defined as \, $x^{-1}\alpha (x)$. The ratio ${\Pr}(G,\Aut(G))$ is called autocommuting probability of $G$.

Let  $H$ and $K$ be two subgroups of a finite group $G$ such that $H\subseteq K$.
Motivated by the works in \cite{DN10, moga},  we define
\begin{equation}\label{GenAutComDeg}
{\Pr}_g(H,\Aut(K)) = \frac {\left|\{(x, \alpha)\in H \times \Aut(K) : [x, \alpha] = g\}\right|}{|H||\Aut(K)|}
\end{equation}
where $g\in K$.
That is,  ${\Pr}_g(H,\Aut(K))$ is the probability  that the autocommutator of a randomly chosen pair of elements, one from $H$ and the other from $\Aut(K)$, is equal to a given element $g \in K$. The ratio  ${\Pr}_g(H,\Aut(K))$ is called generalized
autocommuting probability of $G$ relative to its subgroups $H$ and $K$.  Clearly, if $H = G$ and $g = 1$ then ${\Pr}_g(H,\Aut(K)) = {\Pr}(G,\Aut(G))$. We would like to mention here that the case when $H = G$ is considered in \cite{pd1}. In this paper, we study ${\Pr}_g(H,\Aut(K))$ extensively. In particular, we obtain some computing formulae, various bounds and a few characterizations of $G$ through a subgroup. We conclude the paper describing an invariance property of  ${\Pr}_g(H,\Aut(K))$.

 We write $S(H,\Aut(K))$ to denote the set  $\{[x,\alpha] : x \in H \text{ and } \alpha \in \Aut(K)\}$ and  $[H,\Aut(K)] := \langle S(H,\Aut(K)) \rangle$. We also write
$L(H,\Aut(K)) := \{x \in H : [x, \alpha] = 1 \text{ for all }\alpha \in \Aut(K)\}$ and $L(G) := L(G,\Aut(G))$, the absolute center of $G$ (see \cite{hegarty}). Note that $L(H,\Aut(K))$ is a normal subgroup of $H$ contained in  $H \cap Z(K)$. Further, $L(H,\Aut(K)) = \underset{\alpha \in \Aut(K)}{\cap}C_H(\alpha)$, where   $C_H(\alpha) = \{x\in H:[x,\alpha] = 1\}$ is a subgroup of $H$.
Let $C_{\Aut(K)}(x) := \{\alpha\in \Aut(K) : \alpha(x) = x\}$ for  $x \in H$ and  $C_{\Aut(K)}(H) = \{\alpha \in \Aut(K) : \alpha (x)  = x \text{ for all } x \in H\}$. Then $C_{\Aut(K)}(x)$ is a subgroup of $\Aut(K)$ and  $C_{\Aut(K)}(H) = \underset{x\in H}{\cap}C_{\Aut(K)}(x)$.
}
 Clearly, ${\Pr}_g(H, \Aut(K)) = 1$ if and only if $[H,\Aut(K)] = \{1\}$ and $g = 1$  if and only if $H = L(H,\Aut(K))$ and $g = 1$. Also, ${\Pr}_g(H, \Aut(K)) = 0$ if and only if  $g \notin S(H,\Aut(K))$. Therefore, we consider $H \ne L(H,\Aut(K))$ and $g  \in S(H,\Aut(K))$ throughout the paper.

\section{Some computing formulae}


For any $x\in H$, let us define the set $T_{x,g}(H, K) = \{\alpha \in \Aut(K):[x,\alpha]=g\}$, where $g$ is a fixed element of $K$. Note that $T_{x, 1}(H, K) = C_{\Aut(K)}(x)$. The following two lemmas play  a crucial role in obtaining   computing formula for ${\Pr}_g(H,\Aut(K))$.

\begin{lemma}\label{lemma1}
Let $H$ and $K$ be two subgroups of a finite group $G$ such that $H \subseteq K$. If $T_{x, g}(H, K) \ne \phi$ then $T_{x, g}(H, K) = \sigma C_{\Aut(K)}(x)$ for some $\sigma\in T_{x, g}(H, K)$ and hence $|T_{x, g}(H, K)| = |C_{\Aut(G)}(x)|$.
\end{lemma}
\begin{proof}
Let $\sigma\in T_{x, g}(H, K)$ and $\beta \in \sigma C_{\Aut(K)}(x)$. Then $\beta = \sigma \alpha$ for some $\alpha\in C_{\Aut(K)}(x)$. We have
\[
[x,\beta] = [x,\sigma \alpha] = x^{-1}\sigma(\alpha (x)) = [x,\sigma] = g.
\]
Therefore, $\beta \in T_{x, g}(H, K)$ and so $\sigma  C_{\Aut(K)}(x)  \subseteq T_{x, g}(H, K)$. Again, let $\gamma\in T_{x, g}(H, K)$ then $\gamma(x) = xg$. We have  $\sigma ^{-1} \gamma(x) = \sigma^{-1}(xg) = x$ and so $\sigma^{-1}  \gamma \in C_{\Aut(K)}(x)$. Therefore, $\gamma\in \sigma  C_{\Aut(K)}(x)$ which gives $T_{x, g}(H, K) \subseteq \sigma C_{\Aut(K)}(x)$. Hence, the result follows.
\end{proof}

Consider the action of $\Aut(K)$ on $K$ given by $(\alpha,x)\mapsto \alpha(x)$  where $\alpha \in \Aut(K)$ and $x\in K$.  Let $\orb_K(x) := \{\alpha(x) : \alpha \in \Aut(K)\}$ be the orbit of $x \in K$. Then by orbit-stabilizer theorem, we have
\begin{equation}\label{orbit-stabilizer-Thm}
 |\orb_K(x)| = \frac {|\Aut(K)|}{|C_{\Aut(K)}(x)|}.
\end{equation}
\begin{lemma}\label{lemma01}
Let $H$ and $K$ be two subgroups of a finite group $G$ such that $H \subseteq K$. Then $T_{x, g}(H, K) \ne \phi$ if and only if  $xg \in  \orb_K(x)$.
\end{lemma}
\begin{proof}
The result follows from the fact that  $\alpha \in T_{x, g}(H, K)$ if and only if $xg \in  \orb_K(x)$.
\end{proof}

The following theorem gives two computing formulae for ${\Pr}_g(H,\Aut(K))$.
\begin{theorem}\label{thm2.4}
Let $H$ and $K$ be two subgroups of a finite group $G$ such that $H \subseteq K$. If  $g \in K$ then
\begin{align*}
{\Pr}_g(H,\Aut(K)) = & \frac {1}{|H||\Aut(K)|}\underset{xg\in \orb_K(x)}{\underset{x\in H}{\sum}}|C_{\Aut(K)}(x)|\\
= &\frac {1}{|H|}\underset{xg\in \orb_K(x)}{\underset{x\in H}{\sum}}\frac {1}{|\orb_K(x)|}. 
\end{align*}
\end{theorem}
\begin{proof}
We have $\{(x,\alpha) \in H \times \Aut(K) : [x, \alpha] = g\} = \underset{x \in H}{\sqcup}(\{x\}\times T_{x, g}(H, K))$, where $\sqcup$ represents the union of disjoint sets. Therefore, by \eqref{GenAutComDeg}, we have
\[
|H||\Aut(K)|{\Pr}_g(H,\Aut(K)) =  |\underset{x \in H}{\sqcup}(\{x\}\times T_{x, g}(H, K))| = \underset{x \in H}{\sum}|T_{x, g}(H, K)|.
\]
Hence, the result follows from Lemma \ref{lemma1}, Lemma \ref{lemma01} and \eqref{orbit-stabilizer-Thm}.
\end{proof}

Considering  $g = 1$ in Theorem \ref{thm2.4}, we get the following  computing formulae for ${\Pr}(H,\Aut(K))$.
\begin{corollary}\label{formula2}
Let $H$ and $K$ be two subgroups of a finite group $G$ such that $H \subseteq K$. Then
\[
{\Pr}(H,\Aut(K)) = \frac {1}{|H||\Aut(K)|}\underset{x\in H}{\sum}|C_{\Aut(K)}(x)| = \frac {|\orb_K(H)|}{|H|}
\]
where $\orb_K(H) = \{\orb_K(x) : x \in H\}$.
\end{corollary}

\begin{corollary}
Let $H$ and $K$ be two subgroups of a finite group $G$ such that $H \subseteq K$. If $C_{\Aut (K)}(x)=\{I\}$   for all $x \in H\setminus \{1\}$, where $I$ is the identity element of $\Aut(K)$, then
\[
\Pr(H,\Aut(K)) = \frac {1}{|H|} + \frac {1}{|\Aut(K)|} - \frac {1}{|H||\Aut(K)|}.
\]
\end{corollary}
\begin{proof}
By Corollary \ref{formula2}, we have
\[
|H||\Aut(K)|\Pr(H,\Aut(K))  = \underset{x\in H}{\sum}|C_{\Aut(K)}(x)|
  =   |\Aut(K)| + |H| - 1.
\]
Hence, the result follows.
\end{proof}
\noindent
We also have
$|\{(x,\alpha) \in H \times \Aut(K) : [x, \alpha] = 1\}| = \underset{\alpha \in \Aut(K)}{\sum}|C_H(\alpha)|$ and hence
\begin{equation}
{\Pr}(H,\Aut(K)) = \frac {1}{|H||\Aut(K)|}\underset{\alpha \in \Aut(K)}{\sum}|C_H(\alpha)|.
\end{equation}

We conclude this section with the following two results.
\begin{proposition}
Let $H$ and $K$ be two subgroups of a finite group $G$ such that $H \subseteq K$. If $g \in K$ then
\[
{\Pr}_{g^{-1}}(H, \Aut(K)) = {\Pr}_g(H, \Aut(K)).
\]
\end{proposition}

\begin{proof}
  Let
  \begin{align*}
   A &= \{(x, \alpha) \in H \times \Aut(K) : [x, \alpha] = g\} \text{ and }\\
   B &= \{ (y, \beta) \in H \times \Aut(K) : [y, \beta] = g^{-1} \}.
  \end{align*}
  Then $(x, \alpha) \mapsto ( \alpha(x), \alpha ^{-1})$ gives a bijection between $A$ and $B$. Therefore $|A| = |B|$. Hence, the result follows from \eqref{GenAutComDeg}.

\end{proof}

%

\begin{proposition}
Let $G_1$ and $G_2$ be two finite groups. Let $H_1, K_1$ and $H_2, K_2$ be subgroups of $G_1$ and $G_2$ respectively such that $H_1\subseteq K_1$, $H_2\subseteq K_2$ and $\gcd(|K_1|, |K_2|) = 1$.  If $(g_1, g_2) \in K_1 \times K_2$ then
\[
{\Pr}_{(g_1, g_2)}(H_1 \times H_2, \Aut(K_1 \times K_2)) = {\Pr}_{g_1}(H_1, \Aut(K_1)) {\Pr}_{g_2}(H_2, \Aut(K_2)).
\]
\end{proposition}
\begin{proof}
Let
\begin{align*}
\mathcal{X} &= \{((x, y), \alpha_{K_1 \times K_2}) \in (H_1 \times H_2) \times \Aut(K_1 \times K_2) :\\
&\hspace{5cm} [(x, y), \alpha_{K_1 \times K_2}] = (g_1, g_2)\},\\
\mathcal{Y} &= \{(x, \alpha_{K_1}) \in H_1 \times \Aut(K_1) : [x, \alpha_{K_1}] = g_1\} \text{ and }\\
\mathcal{Z} &= \{(y, \alpha_{K_2}) \in H_2 \times \Aut(K_2) : [y, \alpha_{K_2}] = g_2\}.
\end{align*}
Since  $\gcd(|K_1|,|K_2|) = 1$, by Lemma 2.1 of \cite{cD07},  we have $\Aut(K_1 \times K_2) = \Aut(K_1)$ $\times \Aut(K_2)$. Therefore, for every $\alpha_{K_1 \times K_2} \in \Aut(K_1 \times K_2)$ there exist unique $\alpha_{K_1} \in  \Aut(K_1)$  and $\alpha_{K_2} \in  \Aut(K_2)$ such that $\alpha_{K_1 \times K_2} = \alpha_{K_1} \times \alpha_{K_2}$, where $\alpha_{K_1} \times \alpha_{K_2}((x, y)) = (\alpha_{K_1}(x), \alpha_{K_2}(y))$ for all $(x, y) \in H_1 \times H_2$. Also, for all $(x, y) \in H_1 \times H_2$, we have $[(x, y), \alpha_{K_1 \times K_2}] = (g_1, g_2)$ if and only if $[x, \alpha_{K_1}] = g_1$ and $[y, \alpha_{K_2}] = g_2$. These leads to show that $\mathcal{X} = \mathcal{Y} \times \mathcal{Z}$. Therefore
\[
\frac{|\mathcal{X}|}{|H_1 \times H_2||\Aut(K_1 \times K_2)|}
= \frac{|\mathcal{Y}|}{|H_1||\Aut(K_1)|}\cdot\frac{|\mathcal{Z}|}{|H_2||\Aut(K_2)|}.
\]
Hence, the result follows from \eqref{GenAutComDeg}.
\end{proof}

\section{Various bounds}
In this section, we obtain various bounds for ${\Pr}_g(H,\Aut(K))$. We begin with the following lower bounds.

\begin{proposition}
Let $H$ and $K$ be two subgroups of a finite group $G$ such that $H \subseteq K$. Then, for $g \in K$, we have
\begin{enumerate}
\item  
${\Pr}_g(H,\Aut(K))\geq \frac {|L(H,\Aut(K))|}{|H|} +  \frac {|C_{\Aut(K)}(H)|(|H|-|L(H,\Aut(K))|)}{|H||\Aut(K)|}$ if $g = 1$.
\item  
${\Pr}_g(H,\Aut(K))\geq \frac {|L(H,\Aut(K))||C_{\Aut(K)}(H)|}{|H||\Aut(K)|}$ if $g \neq 1$.
\end{enumerate}
\end{proposition}

\begin{proof}
Let $\mathcal{C}$ denote the set  $\{(x,\alpha) \in H\times \Aut(K): [x,\alpha] = g\}$.

(a) We have $(L(H,\Aut(K)) \times \Aut(K))\cup (H \times C_{\Aut(K)}(H))$ is a subset of $\mathcal{C}$ and
$|(L(H,\Aut(K)) \times \Aut(K))\cup (H \times C_{\Aut(K)}(H))| =|L(H,\Aut(K))||\Aut(K)| + |C_{\Aut(K)}(H)||H| - |L(H,\Aut(K))||C_{\Aut(K)}(H)|$. Hence, the result follows from \eqref{GenAutComDeg}.

(b) Since $g\in S(H,\Aut (K))$ we have $\mathcal{C}$ is non-empty. Let $(y, \beta) \in \mathcal{C}$ then $(y, \beta) \notin  L(H,\Aut(K))\times C_{\Aut(K)}(H)$ otherwise $[y, \beta] = 1$. It is easy to see that the coset $(y, \beta)(L(H,\Aut(K))\times C_{\Aut(K)}(H))$ is a subset of $\mathcal{C}$ having order $|L(H,\Aut(K))||C_{\Aut(K)}(H)|$. Hence, the result follows from \eqref{GenAutComDeg}.
\end{proof}

\begin{proposition}\label{prop3.2}
Let $H$ and $K$ be two subgroups of a finite group $G$ such that $H \subseteq K$. If $g \in K$ then
\[
{\Pr}_g(H,\Aut(K)) \leq \Pr(H,\Aut(K)).
\]
The equality holds if and only if $g = 1$.
\end{proposition}
\begin{proof}
By Theorem \ref{thm2.4}, we have
\begin{align*}
{\Pr}_g(H, \Aut(K)) &= \frac {1}{|H||\Aut(K)|}\underset{xg\in \orb_K(x)}{\underset{x \in H}{\sum}}|C_{\Aut(K)}(x)|\\
&\leq \frac {1}{|H||\Aut(K)|}\underset{x \in H}{\sum}|C_{\Aut(K)}(x)| = \Pr(H,\Aut(K)).
\end{align*}
The equality holds if and only if $xg\in \orb_K(x)$ for all $x \in H$ if and only if $g = 1$.
\end{proof}

\begin{proposition}
Let $H$ and $K$ be two subgroups of a finite group $G$ such that $H \subseteq K$. Let $g \in K$ and $p$ the smallest prime dividing $|\Aut(K)|$. If  $g \neq 1$ then
\[
{\Pr}_g(H, \Aut(K))\leq \frac {|H| - |L(H,\Aut(K))|}{p|H|} < \frac {1}{p}.
\]
\end{proposition}
\begin{proof}
By Theorem \ref{thm2.4}, we have
\begin{equation}\label{eq3.1}
{\Pr}_g(H, \Aut(K)) = \frac {1}{|H|}\underset{xg\in \orb_K(x)}{\underset{x\in H \setminus L(H,\Aut(K))}{\sum}}\frac {1}{|\orb_K(x)|}
\end{equation}
noting that for $x \in L(H,\Aut(K))$ we have $xg \notin \orb_K(x)$. Also, for $x\in H \setminus L(H,\Aut(K))$ and $xg \in \orb_K(x)$ we have $|\orb_K(x)| > 1$. Since $|\orb_K(x)|$ is a divisor of $|\Aut(K)|$ we have $|\orb_K(x)| \geq p$. Hence, the result follows from \eqref{eq3.1}.
\end{proof}

\begin{proposition}
Let $H_1$, $H_2$ and $K$ be subgroups of a finite group $G$ such that $H_1 \subseteq H_2 \subseteq K$. Then 
\[
{\Pr}_g(H_1,\Aut(K))\leq |H_2:H_1|{\Pr}_g(H_2,\Aut(K)).
\]
The equality holds if and only if $xg \notin \orb_K(x)$ for all $x \in H_2\setminus H_1$.
\end{proposition}

\begin{proof}
By Theorem \ref{thm2.4}, we have
\begin{align*}
|H_1||\Aut(K)|{\Pr}_g(H_1,\Aut(K))&= \underset{xg \in \orb_K(x)}{\underset{x\in H_1}{\sum}}|C_{\Aut(K)}(x)|\\
&\leq \underset{xg \in \orb_K(x)}{\underset{x\in H_2}{\sum}}|C_{\Aut(K)}(x)|\\
&=|H_2||\Aut(K)|{\Pr}_g(H_2,\Aut(K)).
\end{align*}
Hence, the result follows.
\end{proof}

\begin{proposition}
Let $H$ and $K$ be two subgroups of a finite group $G$ such that $H \subseteq K$. If $g\in K$ then
\[
{\Pr}_g(H,\Aut(K))\leq |K : H|\Pr(K,\Aut(K))
\]
with  equality if and only if $g = 1$ and $H = K$.
\end{proposition}

\begin{proof}
By Proposition \ref{prop3.2}, we have
\begin{align*}
{\Pr}_g(H,\Aut(K)) & \leq \Pr(H,\Aut(K))\\
&=\frac {1}{|H||\Aut(K)|}\underset{x\in H}{\sum}|C_{\Aut(K)}(x)|\\
&\leq \frac {1}{|H||\Aut(K)|}\underset{x\in K}{\sum}|C_{\Aut(K)}(x)|.
\end{align*}
Hence, the result follows from Corollary \ref{formula2}.
\end{proof}

Note that if we replace $\Aut(K)$ by $\Inn(K)$, the inner automorphism group of $K$, in \eqref{GenAutComDeg} then ${\Pr}_g(H,\Inn(K)) = {\Pr}_g(H,K)$ where 
\[
{\Pr}_g(H,K) = \frac{|\{(x, y) \in H \times K : x^{-1}y^{-1}xy = g\}|}{|H||K|}.
\]
A detailed study on ${\Pr}_g(H,K)$ can be found in \cite{DN10}. The following proposition gives a relation between ${\Pr}_g(H,\Aut(K))$ and ${\Pr}_g(H,K)$ for $g = 1$.
\begin{proposition}
Let $H$ and $K$ be two subgroups of a finite group $G$ such that $H \subseteq K$. If $g = 1$ then 
\[
{\Pr}_g(H,\Aut(K)) \leq {\Pr}_g(H,K).
\]
\end{proposition}
\begin{proof}
If $g = 1$ then by  \cite[Theorem 2.3]{DN10}, we have 
\begin{equation}\label{newbound01}
{\Pr}_g(H, K) =   \frac {1}{|H|}\underset{x\in H}{\sum}\frac {1}{|\cl_K(x)|}
\end{equation}
where $\cl_K(x) = \{\alpha(x) : \alpha \in \Inn(K)\}$. Since $\cl_K(x) \subseteq \orb_K(x)$ for all $x \in H$, the result follows from \eqref{newbound01} and Theorem \ref{thm2.4}.     
\end{proof}
 
\begin{theorem}\label{thm3.6}
Let $H$ and $K$ be two subgroups of a finite group $G$ such that $H \subseteq K$ and $p$ the smallest prime dividing $|\Aut(K)|$.  Then
\[
\Pr(H,\Aut(K)) \geq \frac {|L(H,\Aut(K))|}{|H|} + \frac {p(|H| - |X_H| - |L(H,\Aut(K))|) + |X_H|}{|H||\Aut(K)|}
\]
and
\[
\Pr(H,\Aut(K)) \leq \frac {(p - 1)|L(H,\Aut(K))| + |H|}{p|H|} - \frac {|X_H|(|\Aut(K)| - p)}{p|H||\Aut(K)|},
\]
where $X_H = \{x \in H : C_{\Aut(K)}(x) = \{I\}\}$.
\end{theorem}
\begin{proof}
We have   $X_H \cap L(H,\Aut(K)) = \phi$. Therefore
\begin{align*}
\underset{x\in H}{\sum}|C_{\Aut(K)}(x)| = & \, |X_H| + |\Aut(K)||L(H,\Aut(K))| \\
& + \underset{x\in H\setminus (X_H\cup L(H,\Aut(K)))}{\sum}|C_{\Aut(K)}(x)|.
\end{align*}

For $x  \in H \setminus (X_H \cup L(H,\Aut(K)))$ we have $\{I\}\neq C_{\Aut(K)}(x)\neq \Aut(K)$ which implies $p \leq |C_{\Aut(K)}(x)|\leq \frac {|\Aut(K)|}{p}$. Therefore
\begin{align}\label{our_bound1}
\underset{x\in H}{\sum}|C_{\Aut(K)}(x)| \geq & |X_H| + |\Aut(K)||L(H,\Aut(K))|\nonumber \\
& + p(|H| - |X_H| - |L(H,\Aut(K))|)
\end{align}
and
\begin{align}\label{our_bound2}
\underset{x\in H}{\sum}|C_{\Aut(K)}(x)| \leq & |X_H| + |\Aut(K)||L(H,\Aut(K))| \nonumber\\
& + \frac{|\Aut(K)|(|H| - |X_H| - |L(H,\Aut(K))|)}{p}.
\end{align}
Hence, the result follows from Corollary \ref{formula2}, \eqref{our_bound1} and \eqref{our_bound2}.
\end{proof}

\noindent We have the following two corollaries.
\begin{corollary}\label{bound_like3/4}
Let $H$ and $K$ be two subgroups of a finite group $G$ such that $H \subseteq K$. If $p$ and $q$ are the smallest primes dividing  $|\Aut(K)|$ and $|H|$ respectively then
\[
\Pr(H,\Aut(K)) \leq \frac{p + q - 1}{pq}.
\]
In particular, if $p = q$ then $\Pr(H,\Aut(K)) \leq \frac{2p - 1}{p^2} \leq \frac{3}{4}$.
\end{corollary}
\begin{proof}
Since $H \ne L(H,\Aut(K))$ we have $|H : L(H,\Aut(K))| \geq q$. Therefore, by Theorem \ref{thm3.6}, we have
\[
\Pr(H,\Aut(K)) \leq \frac{1}{p}\left(\frac{p - 1}{|H : L(H,\Aut(K))|} + 1\right) \leq \frac{p + q - 1}{pq}.
\]
\end{proof}

\begin{corollary}\label{bound_like5/8}
Let $H$ and $K$ be two subgroups of a finite group $G$ such that $H \subseteq K$ and $p$, $q$ be the smallest primes dividing  $|\Aut(K)|$ and $|H|$ respectively. If $H$ is non-abelian then
\[
\Pr(H,\Aut(K)) \leq \frac{q^2 + p - 1}{pq^2}.
\]
In particular, if $p = q$ then $\Pr(H,\Aut(K)) \leq \frac{p^2 + p - 1}{p^3} \leq \frac{5}{8}$.
\end{corollary}
\begin{proof}
Since $H$ is non-abelian we have $|H : L(H,\Aut(K))| \geq q^2$. Therefore, by Theorem \ref{thm3.6}, we have
\[
\Pr(H,\Aut(K)) \leq \frac{1}{p}\left(\frac{p - 1}{|H : L(H,\Aut(K))|} + 1\right) \leq \frac{q^2 + p - 1}{pq^2}.
\]
\end{proof}
Now we obtain two lower bounds analogous to the lower bounds obtained in  \cite[Theorem A]{nY15} and \cite[Theorem 1]{ND10}.
\begin{theorem}\label{prop3.8}
Let $H$ and $K$ be two subgroups of a finite group $G$ such that $H \subseteq K$. Then
\[
\Pr(H,\Aut(K)) \geq \frac {1}{|S(H,\Aut(K))|}\left(1 + \frac {|S(H,\Aut(K))| - 1}{|H : L(H,\Aut(K))|}\right).
\]
The equality holds if and only if $\orb_K(x) = xS(H,\Aut(K))$
 for all $x \in H \setminus L(H,\Aut(K))$.
\end{theorem}
\begin{proof}
For all $x \in H \setminus L(H,\Aut(K))$ we have $\alpha (x) = x[x, \alpha] \in xS(H,\Aut(K))$. Therefore $\orb_K(x) \subseteq xS(H,\Aut(K))$ and so $|\orb_K(x)| \leq |S(H,\Aut(K))|$
 for all $x \in H \setminus L(H,\Aut(K))$. Now, by Corollary \ref{formula2}, we have
\begin{align*}
\Pr(H,\Aut(K))& = \frac {1}{|H|}\left(\underset{x \in L(H,\Aut(K))}{\sum}\frac {1}{|\orb_K(x)|} + \underset{x \in H \setminus L(H,\Aut(K))}{\sum}\frac {1}{|\orb_K(x)|}\right)\\
&\geq \frac {|L(H,\Aut(K))|}{|H|} + \frac {1}{|H|}\underset{x\in H\setminus L(H,\Aut(K))}{\sum}\frac{1}{|S(H,\Aut(K))|}.
\end{align*}
Hence, the result follows.
\end{proof}
\begin{lemma}\label{lemma4.4}
Let $H$ and $K$ be two subgroups of a finite group $G$ such that $H \subseteq K$. Then, for any two integers  $m \geq n$, we have
\[
\frac {1}{n}\left(1 + \frac{n - 1}{|H : L(H,\Aut(K))|}\right) \geq \frac {1}{m}\left(1 + \frac{m - 1}{|H : L(H,\Aut(K))|}\right).
\]
 If $L(H,\Aut(K))\neq H$ then equality holds if and only if $m=n$.
\end{lemma}
\begin{proof}
The proof is an easy exercise.
\end{proof}

\begin{corollary}\label{lastcor}
Let $H$ and $K$ be two subgroups of a finite group $G$ such that $H \subseteq K$. Then
\[
{\Pr}(H,\Aut(K))\geq \frac {1}{|[H,\Aut(K)]|}\left(1 + \frac {|[H,\Aut(K)]| - 1}{|H : L(H,\Aut(K))|}\right).
\]
If $H \ne L(H,\Aut(K))$ then the equality holds if and only if $[H,\Aut(K)] =  S(H,\Aut(K))$ and $\orb_K(x)= x[H,\Aut(K)]$
 for all $x \in H \setminus L(H,\Aut(K))$.
\end{corollary}
\begin{proof}
Since $|[H,\Aut(K)]|  \geq |S(H, \,\Aut(K))|$, the result follows from Theorem \ref{prop3.8} and Lemma \ref{lemma4.4}.

Note that the equality holds if and only if equality holds
in Theorem \ref{prop3.8} and Lemma \ref{lemma4.4}.
\end{proof}
It is worth mentioning that  Theorem \ref{prop3.8} gives better  lower bound  than the lower bound given by Corollary \ref{lastcor}. Also
\begin{align*}
 \frac {1}{|[H,\Aut(K)]|}\left(1 + \frac {|[H,\Aut(K)]| - 1}{|H : L(H,\Aut(K))|}\right)
\geq & \frac
{|L(H,\Aut(K))|}{|H|} \\
& + \frac {p(|H| - |L(H,\Aut(K))|)}{|H||\Aut(K)|}.
\end{align*}
Hence, Theorem \ref{prop3.8} gives better lower bound than the lower bound given by Theorem \ref{thm3.6}.

\section{A few Characterizations}
 In this section, we obtain some characterizations of a subgroup $H$ of $G$ if equality holds in Corollary \ref{bound_like3/4} and Corollary \ref{bound_like5/8}. We begin with the following result.
\begin{theorem}\label{prop4.1}
Let $H$ and $K$ be two subgroups of a finite group $G$ such that $H \subseteq K$. If $\Pr(H,\Aut(K)) = \frac {p + q -1}{pq}$ for some primes $p$ and $q$. Then  $pq$  divides $|H||\Aut(K)|$. Further, if $p$ and $q$ are the smallest primes dividing $|\Aut(K)|$ and $|H|$ respectively, then
\[
\frac{H}{L(H,\Aut(K))} \cong \mathbb Z_q.
\]
In particular, if $H$ and $\Aut (K)$ are of even order and $\Pr(H,\Aut(K)) = \frac{3}{4}$ then $\frac{H}{L(H,\Aut(K))}\cong\mathbb Z_2$.
\end{theorem}
\begin{proof}
By \eqref{GenAutComDeg}, we have $(p + q -1)|H||\Aut(K)| = pq|\{(x,\alpha)\in H\times \Aut(K):[x,\alpha]=1\}|$. Therefore, $pq$ divides $|H||\Aut(K)|$.

If $p$ and $q$ are the smallest primes dividing $|\Aut(K)|$ and $|H|$ respectively then, by Theorem \ref{thm3.6}, we have
\[
\frac{p + q -1}{pq} \leq \frac{1}{p}\left(\frac{p - 1}{|H : L(H,\Aut(K))|} + 1\right)
\]
which gives $|H : L(H,\Aut(K))| \leq q$. Hence, $\frac{H}{L(H,\Aut(K))} \cong \mathbb Z_q$.
\end{proof}

\begin{theorem}\label{prop4.2}
Let $H \subseteq K$ be two subgroups of a finite group $G$ such that  $H$ is non-abelian and $\Pr(H,\Aut(K)) = \frac {q^2 + p - 1}{pq^2}$ for some primes $p$ and $q$. Then $pq$ divides  $|H||\Aut(K)|$. Further, if $p$ and $q$ are the smallest primes dividing $|\Aut(K)|$ and $|H|$ respectively then
\[
\frac{H}{L(H,\Aut(K))}\cong\mathbb Z_q \times \mathbb Z_q.
\]
In particular, if $H$ and $\Aut (K)$ are of even order and $\Pr(H,\Aut(K)) = \frac{5}{8}$ then $\frac{H}{L(H,\Aut(K))} \cong \mathbb Z_2 \times \mathbb Z_2$.
\end{theorem}

\begin{proof}
By \eqref{GenAutComDeg}, we have $(q^2 + p -1)|H||\Aut(K)| = pq^2|\{(x,\alpha)\in H\times \Aut(K):[x,\alpha]=1\}|$. Therefore, $pq$ divides $|H||\Aut(K)|$.

If $p$ and $q$ are the smallest primes dividing $|\Aut(K)|$ and $|H|$ respectively then, by Theorem \ref{thm3.6}, we have
\[
\frac{q^2 + p -1}{pq^2} \leq \frac{1}{p}\left(\frac{p - 1}{|H : L(H,\Aut(K))|} + 1\right)
\]
which gives $|H : L(H,\Aut(K))| \leq q^2$. Since $H$ is non-abelian we have $|H : L(H,\Aut(K))| \neq 1, q$. Hence, $\frac{H}{L(H,\Aut(K))} \cong \mathbb Z_q \times \mathbb Z_q$.
\end{proof}


The following two results give partial converses of Theorem \ref{prop4.1} and \ref{prop4.2} respectively.

\begin{proposition}
Let $H$ and $K$ be two subgroups of a finite group $G$ such that $H \subseteq K$. Let $p, q$ be the smallest prime divisors of $|\Aut(K)|$, $|H|$ respectively and $|\Aut(K) : C_{\Aut(K)}(x)| = p$ for all $x \in H \setminus L(H,\Aut(K))$.
\begin{enumerate}
\item
If $\frac{H}{L(H,\Aut(K))} \cong \mathbb Z_q$ then $\Pr(H,\Aut(K)) = \frac {p + q - 1}{pq}$.
\item
If $\frac{H}{L(H,\Aut(K))} \cong\mathbb Z_q \times \mathbb Z_q$  then $\Pr(H,\Aut(K)) = \frac {q^2 + p - 1}{pq^2}$.
\end{enumerate}
\end{proposition}
\begin{proof}
Since $|\Aut(K) : C_{\Aut(K)}(x)| = p$ for all $x \in H \setminus L(H,\Aut(K))$ we have $|C_{\Aut(K)}(x)| = \frac{|\Aut(K)|}{p}$ for all $x \in H \setminus L(H,\Aut(K))$. Therefore, by Corollary \ref{formula2}, we have
\begin{align*}
\Pr(H,\Aut(K)) &= \frac{|L(H,\Aut(K))|}{|H|} + \frac {1}{|H||\Aut(K)|} \underset{x \in H \setminus L(H,\Aut(K))}{\sum}|C_{\Aut(K)}(x)|\\
&= \frac {|L(H,\Aut(K))|}{|H|} + \frac{|H| - |L(H,\Aut(K))|}{p|H|}.
\end{align*}
Thus
\begin{equation}\label{conv_eq}
\Pr(H,\Aut(K)) = \frac{1}{p}\left(\frac{p - 1}{|H : L(H,\Aut(K))|}  + 1\right).
\end{equation}
(a) If $\frac{H}{L(H,\Aut(K))} \cong \mathbb Z_q$ then \eqref{conv_eq} gives $\Pr(H,\Aut(K)) = \frac {p + q - 1}{pq}$.

\noindent (b) If $\frac{H}{L(H,\Aut(K))} \cong\mathbb Z_q \times \mathbb Z_q$  then \eqref{conv_eq} gives $\Pr(H,\Aut(K)) = \frac {q^2 + p - 1}{pq^2}$.
\end{proof}

\section{Autoisoclinic pairs}
In the year 1940, Hall \cite{pH40} introduced the concept of  isoclinism between two groups.
Following Hall, Moghaddam et al.\ \cite{msE13} have defined autoisoclinism between two groups, in the year 2013. Recall that two groups $G_1$ and $G_2$ are said to be autoisoclinic  if there exist isomorphisms $\psi : \frac{G_1}{L(G_1)} \to \frac{G_2}{L(G_2)}$, $\beta : [G_1, \Aut(G_1)] \to [G_2, \Aut(G_2)]$ and $\gamma : \Aut(G_1) \to \Aut(G_2)$ such that the following diagram commutes
\begin{center}
$
\begin{CD}
   \frac{G_1}{L(G_1)} \times \Aut(G_1) @>\psi \times \gamma>> \frac{G_2}{L(G_2)} \times
   \Aut(G_2)\\
   @VV{a_{(G_1, \Aut(G_1))}}V  @VV{a_{(G_2, \Aut(G_2))}}V\\
   [G_1, \Aut(G_1)] @> \beta >> [G_2, \Aut(G_2)]
\end{CD}
$
\end{center}
where the maps $a_{(G_i, \Aut(G_i))}: \frac{G_i}{L(G_i)} \times \Aut(G_i) \to [G_i, \Aut(G_i)]$, for  $i = 1, 2$,
 are given by
\[
a_{(G_i, \Aut(G_i))}(x_iL(G_i), \alpha_i) = [x_i, \alpha_i].
\]
 Such a pair $(\psi\times\gamma, \beta)$ is called an autoisoclinism between the groups $G_1$ and $G_2$.
We generalize the notion of autoisoclinism in the following way:

Let  $H_1, K_1$ and $H_2, K_2$ be subgroups of the groups $G_1$ and $G_2$  respectively.   The pairs of subgroups $(H_1, K_1)$ and $(H_2,
K_2)$ such that $H_1 \subseteq
K_1$ and $H_2 \subseteq K_2$ are said to be autoisoclinic  if there exist isomorphisms $\psi : \frac{H_1}{L(H_1, \Aut{K_1})} \to \frac{H_2}{L(H_2, \Aut(K_2))}$, $\beta :
[H_1, \Aut(K_1)] \to [H_2, \Aut(K_2)]$ and $\gamma : \Aut(K_1) \to \Aut(K_2)$ such that the following diagram commutes

\vspace{.25cm}
\begin{center}
$
\begin{CD}
   \frac{H_1}{L(H_1, \Aut(K_1))} \times \Aut(K_1) @>\psi \times \gamma>> \frac{H_2}{L(H_2, \Aut(K_2))} \times
   \Aut(K_2)\\
   @VV{a_{(H_1, \Aut(K_1))}}V  @VV{a_{(H_2, \Aut(K_2))}}V\\
   [H_1, \Aut(K_1)] @> \beta >> [H_2, \Aut(K_2)]
\end{CD}
$
\end{center}
\noindent where the maps $a_{(H_i, \Aut(K_i))}: \frac{H_i}{L(H_i, \Aut(K_i))} \times \Aut(K_i) \to (H_i, \Aut(K_i))$, for $i = 1,
2$, are given by
\[
a_{(H_i, \Aut(K_i))}(x_iL(H_i, \Aut(K_i)), \alpha_i) = [x_i, \alpha_i].
\]
Such a pair $(\psi\times\gamma, \beta)$ is said to be an autoisoclinism between the pairs of groups $(H_1, K_1)$
and $(H_2, K_2)$. We conclude this paper with the following   generalization of  \cite[Theorem 5.1]{pd1} and \cite[Lemma 2.5]{Rismanchian15}.
\begin{theorem}
Let $G_1$ and $G_2$ be two finite groups with subgroups $H_1, K_1$ and $H_2, K_2$ respectively such that $H_1 \subseteq K_1$ and $H_2 \subseteq K_2$. If $(\psi\times\gamma, \beta)$ is an autoisoclinism between the pairs $(H_1, K_1)$ and
$(H_2, K_2)$ then, for $g \in K_1$,
\[
{\Pr}_g(H_1, \Aut(K_1))
= {\Pr}_{\beta (g)}(H_2, \Aut(K_2)).
\]
\end{theorem}

\begin{proof}
Let us consider the sets $\mathcal{S}_g = \{(x_1L(H_1,\Aut(K_1)),\alpha_1)\in \frac {H_1}{L(H_1,\Aut(K_1))}\times
\Aut(K_1):[x_1L(H_1,\Aut(K_1)),\alpha_1]=g\}$ and $\mathcal{T}_{\beta (g)}=\{(x_2L(H_2,\Aut(K_2)),\alpha_2)\in \frac
{H_2}{L(H_2,\Aut(K_2))}\times \Aut(K_2):[x_2L(H_2,\Aut(K_2)),\alpha_2] = \beta (g)\}$. Since  $(H_1, K_1)$ is 
autoisoclinic to $(H_2, K_2)$ we have $|\mathcal{S}_g| = |\mathcal{T}_{\beta (g)}|$. Again, it is clear that
\begin{equation}\label{auto-eq5.1}
|\{(x_1,\alpha_1)\in H_1\times \Aut(K_1):
[x_1,\alpha_1] = g\}|=|L(H_1,\Aut(K_1))||\mathcal{S}_g| 
\end{equation}
and
\begin{equation}\label{auto-eq5.2}
|\{(x_2,\alpha_2)\in H_2\times \Aut(K_2):
[x_2,\alpha_2]= \beta(g)\}|=|L(H_2,\Aut(K_2))||\mathcal{T}_{\beta (g)}|.
\end{equation}
Hence, the result follows from \eqref{GenAutComDeg}, \eqref{auto-eq5.1} and \eqref{auto-eq5.2}.
\end{proof}

\end{document}